\documentclass[12pt]{article}
\usepackage[margin=1in]{geometry}
\usepackage{amsmath,amsthm,amssymb}
\usepackage{verbatim}
\usepackage{tikz}
\usepackage{mathtools}
\usepackage{hyperref}
\usepackage{enumitem}
\usepackage[mathlines]{lineno}
\usepackage{booktabs}

\tikzstyle{vertex}=[inner sep=2pt,thick,shape=circle,draw=black,fill=white]
\tikzstyle{Vertex}=[inner sep=1pt,thick,shape=circle,draw=black,fill=white]
\tikzstyle{Edge}=[thick]
\tikzstyle{tiny_vertex}=[draw,thick,fill=white,circle,inner sep=1.5pt]
\tikzstyle{full}=[draw,thick,fill=black,circle,inner sep=2pt]
\tikzstyle{empty}=[draw,color=black!40!white,thick,fill=white,circle,inner sep=2pt]

\newtheorem{thm}{Theorem}[section]
\newtheorem{cor}[thm]{Corollary}
\newtheorem{lem}[thm]{Lemma}
\newtheorem{prop}[thm]{Proposition}
\newtheorem{conj}[thm]{Conjecture}

\theoremstyle{definition}

\newtheorem{defn}[thm]{Definition}

\allowdisplaybreaks

\definecolor{forestgreen}{rgb}{0.13, 0.55, 0.13}

\DeclareMathOperator{\dist}{dist}
\DeclareMathOperator{\spec}{spec}

\DeclareMathOperator{\diam}{diam}
\DeclareMathOperator*{\coal}{\circ}

\newcommand{\Dq}{\mathcal{D}_q}
\newcommand{\Df}{\mathcal{D}_f}
\newcommand{\D}{\mathcal{D}}

\title{A Cospectral Construction for the Generalized Distance Matrix}
\author{Ori Friesen\footnote{Macalester College, Saint Paul, MN, USA \texttt{ofriesen@macalester.edu}}
\and Cecily Kolko\footnote{Smith College, Northampton, MA, USA \texttt{ckolko@smith.edu}}
\and Nick Layman\footnote{Iowa State University, Ames, IA, USA \texttt{nlayman@iastate.edu}}
\and Kate Lorenzen\footnote{Linfield University, McMinnville, OR, USA \texttt{klorenzen@linfield.edu}}
\and Sarah Zaske \footnote{Grand Valley State University, Allendale, MI, USA \texttt{zaskes@mail.gvsu.edu}}
\and Amy Zeigler\footnote{Augustana College, Rock Island, IL, USA
\texttt{amyzeigler21@augustana.edu}}}
\date{\today}

\begin{document}
\maketitle 

\begin{abstract}
     The generalized distance matrix of a graph is a matrix in which the $(i,j)$th entry is a function, $f$, of the distance between vertex $i$ and vertex $j$. Depending on the choice of $f$, this family of matrices includes both the adjacency matrix and the traditional distance matrix. We present a cospectral construction for the generalized distance matrix akin to Godsil-McKay Switching. We also investigate a special case of the generalized distance matrix: the exponential distance matrix, which is a matrix where every entry is a value $q$ raised to the power of the distance between the vertices.  We give an upper bound on the values of $q$ needed to show a pair of graphs is cospectral for all values of $q$ corresponding to the diameter of the graphs.  
     We also give cospectral constructions unique to value $q=1/2$.
\end{abstract}

\section{Introduction}
A graph $G$ is a set of vertices, $V(G)$, and a set of edges, $E(G)$, which connect two vertices to one another. There are many ways to represent a graph with a matrix. For example, in an \emph{adjacency matrix} $A$, we would let $A_{i,j} = 1$ if two vertices share an edge, or are \emph{adjacent}, and $A_{i,j} = 0$ otherwise. We can determine many structural properties of a graph by looking at the \emph{spectrum}, the multiset of eigenvalues of the matrix. The connection between a graph structure and the eigenvalues of its matrix representation is the basis of spectral graph theory.  

A large question in spectral graph theory, initially thought to be true, is \textit{Can we determine a graph from its eigenvalues?} However, not all graphs can be determined by their spectrum. For example, the two graphs in Figure \ref{fig:GM-Switching} are a non-isomorphic pair with the same spectrum. We will call such graphs  \emph{cospectral}. 

\begin{figure}[ht]
    \centering
    \begin{tabular}{cc}
       \begin{tikzpicture}
    \node[vertex] (b1) at ({cos(180)},{sin(180)}) {};
\node[vertex] (b2) at ({cos(135)},{sin(135)}) {};
\node[vertex] (b3) at ({cos(90)},{sin(90)}) {};
\node[vertex] (b4) at ({cos(45)},{sin(45)}) {};
\node[vertex] (b5) at ({cos(0)},{sin(0)}) {};
\node[vertex] (b6) at ({cos(-45)},{sin(-45)}) {};
\node[vertex] (b7) at ({cos(-90)},{sin(-90)}) {};
\node[vertex] (b8) at ({cos(-135)},{sin(-135)}) {};

\node[vertex] (a1) at (0,0) {$v$};

\draw[thick] (b1)--(b2)--(b3)--(b4)--(b5)--(b6)--(b7)--(b8)--(b1);
\draw[thick] (b1)--(a1)--(b2);
\draw[thick] (b8)--(a1)--(b5);
\end{tikzpicture}  & \begin{tikzpicture}
\node[vertex] (b1) at ({cos(180)},{sin(180)}) {};
\node[vertex] (b2) at ({cos(135)},{sin(135)}) {};
\node[vertex] (b3) at ({cos(90)},{sin(90)}) {};
\node[vertex] (b4) at ({cos(45)},{sin(45)}) {};
\node[vertex] (b5) at ({cos(0)},{sin(0)}) {};
\node[vertex] (b6) at ({cos(-45)},{sin(-45)}) {};
\node[vertex] (b7) at ({cos(-90)},{sin(-90)}) {};
\node[vertex] (b8) at ({cos(-135)},{sin(-135)}) {};

\node[vertex] (a1) at (0,0) {$v$};

\draw[thick] (b1)--(b2)--(b3)--(b4)--(b5)--(b6)--(b7)--(b8)--(b1);
\draw[thick] (b3)--(a1)--(b4);
\draw[thick] (b7)--(a1)--(b6);
    \end{tikzpicture} \\
       (a)  &  (b)
    \end{tabular}
    \caption{Two graphs that are cospectral for the adjacency matrix by Godsil-McKay switching with switching set $D=\{v\}$ and the remaining vertices form a regular graph which is an equitable partition of size one.}
    \label{fig:GM-Switching}
\end{figure}
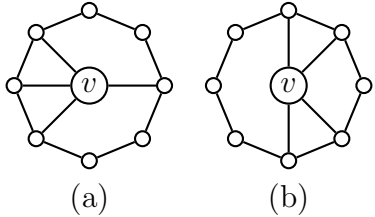

The existence of cospectral graphs can be understood in one of two ways. The first is that two graphs will sometimes happen to share a spectrum because of some \textit{algebraic accident}. If this was the case, we would expect that the differences between a pair of cospectral graphs would be random from one pair to another. The other, more prevailing belief is that these graphs and their structural differences reveal a weakness of the spectrum and can be described with a structural pattern. These patterns are described with cospectral graph constructions. 

One of the more famous constructions of cospectral graphs is by Godsil and McKay \cite{gmckay}, which we will refer to as Godsil-McKay Switching. In their construction, they partition the graph into an equitable partition and a switching set. When certain degree requirements are met, the original graph and the graph resulting from \emph{switching} (changing edges to non-edges, and non-edges to edges) between the equitable partition and the switching set will be cospectral. For example, see Figure \ref{fig:GM-Switching}. Godsil-McKay switching was originally proved for the adjacency matrix, but versions of it have been extended to other matrices; namely, Haemers and Spence in \cite{HS95} extended switching to any matrix of the form $M=\alpha I + \beta A$, which includes the combinatorial Laplacian and the normalized Laplacian. 

There are ways to represent graphs with other matrices beyond the adjacency matrix. 
 The \emph{distance matrix} entries store the distance, the shortest path between two vertices $i,j$, denoted $\dist(i,j)$. Unlike the adjacency matrix, which is a $0-1$ matrix, the distance matrix is dense. Many of the combinatorial techniques used to prove spectral information about the adjacency matrix are more difficult or impossible for the distance matrix, which has seen renewed interest recently in terms of cospectral constructions. In particular, Heysse in \cite{distanceconstruction} showed a cospectral construction using local switching between one vertex and a special set of four vertices. Lorenzen in \cite{cousinvertices} showed a switching-like method between two sets of vertices. 

In addition to the distance matrix, Bin Mahamud et. al. in \cite{coalescing} defined the \emph{generalized distance matrix} $\Df$ as the matrix where $(\Df)_{i,j} = f(\dist(i,j))$ for any function $f$. 
This matrix captures many variations of the distance matrix. In \cite{coalescing}, the authors also presented a cospectral construction for gluing small graphs together called \emph{coalescing}. This mirrors gluing techniques previously developed for the adjacency matrix \cite{Schwenk} and the signless Laplacian \cite{Butler22}. 

We will show sufficient graph structural properties so that two graphs are cospectral for $\Df$, mirroring Godsil-McKay Switching. This expands some of Heysse's work in \cite{distanceconstruction} to the generalized distance matrix and explicitly finds a similarity matrix for these cospectral graphs. This construction unifies and demonstrates that many graphs that are cospectral for one matrix representation are cospectral for many well-studied matrix representations. 

One of the variations of the distance matrix that $\Df$ encapsulates is the \emph{exponential distance matrix} $\Dq$, which is a matrix such that $(\Dq)_{i,j} = q^{\dist(i,j)}$. This matrix was originally introduced by Bapat, Lal, and Pait in \cite{BLP06} where they showed that like the distance matrix, all trees have the same determinant. Many of the known cospectral constructions for distance or variations of distance matrices have not extended to this matrix (except when $\diam(G)=2$) because of the exponential entries. 

In addition to our cospectral construction for $\Df$, we provide a sufficient number of values of $q$ to extend cospectrality to an arbitrary function. Further, we provide a construction of a family of graphs that are only cospectral for $q=\frac12$. Finally, we list some open problems in relation to cospectral constructions and the exponential distance matrix. 

\subsection{Definitions and Notation}
Let $G$ be a graph with vertex set $V(G)$ and edge set $E(G)$. We say that two vertices are adjacent, or neighbors, denoted $v \sim u$, if $\{v,u\}\in E(G)$. The set of all vertices adjacent to a vertex $v$ is called the neighborhood of $v$, $N(v)$. The \emph{degree} of $v$, $\deg(v)$, is the number of times $v$ is an endpoint of an edge. Note: $|N(v)|=\deg(v)$. A \emph{walk} is a sequence of vertices where each consecutive pair is adjacent. A walk with no repeated vertices is called a \emph{path}. The shortest walk between two vertices is called the \emph{distance}, denoted $\dist(u,v)$. The largest distance in a graph is called the \emph{diameter}, mainly $\diam(G)=\max \{\dist(u,v)\}$.

Given two graphs $H$ and $G$, we say that $H$ is a \emph{subgraph} of $G$ if a copy of $H$ can be found in $G$. A path graph $P_n$, where $n$ is the number of vertices of the graph. We say a graph is $k$-regular if every vertex has exactly degree $k$. 

Let $M^G$ be a matrix representation of graph $G$. The \emph{characteristic polynomial} of $M^G$, denoted $\mathcal{P}(x)=\det(xI-M^G)$, is the polynomial whose roots are the multiset of eigenvalues of $M^G$; this is also called the spectrum, $\spec(M^G)$.  Two matrices $A$, $B$ are similar if there exists a \emph{similarity matrix} $P$ such that $P^{-1}AP = B$; moreover, similar matrices have the same spectrum. We define $J_n$ to be an $n \times n$ all-ones matrix.

The distance matrix, $\D^G$, of a graph $G$, is the matrix such that $(\D^G)_{i,j}=\dist(i,j)$. The generalized distance matrix, $\Df^G$ of a graph $G$ is the matrix such that $(\Df^G)_{i,j}=f(\dist(i,j))$ for any function $f: \mathbb{R}\to \mathbb{R}$ (not necessarily continuous). The exponential distance matrix, $\Dq^G$, of a graph $G$ is the matrix such that $(\Dq^G)_{i,j}=q^{\dist(i,j)}$ for some $q$. We will assume that $0\leq q \leq 1$, but there are some interesting examples of cospectral graphs when $q$ is a complex or larger number (see \cite{expdist}).

\section{Godsil-McKay Switching for the Generalized Distance Matrix}

Godsil-McKay Switching is one of the first and most prevalent cospectral constructions for the adjacency matrix. For a graph $G$, we partition the vertices $V(G)=A_1 \cup A_2 \cup \cdots \cup A_k \cup D$ where $A_1, \ldots, A_k$ form an equitable partition and every vertex in $D$ is adjacent to half, all, or none of the vertices in each $A_i$. For example, in Figure \ref{fig:GM-Switching}(a), the switching set $D=\{v\}$ and the other vertices are $A_1$, inducing a regular graph. Notice that $v$ is adjacent to half of the vertices of $A_1$. By \emph{switching} the adjacencies between vertices in $D$ that are adjacent to half the vertices in $A_i$, we create a graph that has the same spectrum as the original. See Figure \ref{fig:GM-Switching}(b) for the cospectral mate of (a). The proof of this switching method introduces a similarity matrix between the two graph matrices. 

Godsil-McKay Switching explains many pairs of cospectral graphs for the adjacency matrix \cite{HS95}. It has also been extended to other matrices such as the Laplacian and normalized Laplacian. Except when the diameter of a graph is at most two or is distance regular, the construction has not previously been extended to distance matrices. 

In Section \ref{subsec:cospecconst}, we present a cospectral construction for the distance matrix using the Godsil-McKay similarity matrix. We'll call this matrix $\hat{S}$, and our proof follows their style. We'll then extend our construction to the generalized distance matrix. 

\subsection{A Cospectral Construction for the Distance Matrix}
\label{subsec:cospecconst}
Let us begin with some linear algebra results.

\begin{lem}

Let $\hat{S} = \frac{2}{n}J_{n} - I_{n}$.
Let $\Vec{x}$ be the $n-$dimensional vector with $\frac{n}{2}$ values being $a$ and the other $\frac{n}{2}$ values being $b$ and $\vec{y}=(a+b)\vec{1}-\vec{x}$ . 
    Then $\hat{S}\Vec{x}=\vec{y}$ .
    Similarly, $\Vec{x}^T \hat{S} = \vec{y}^T$.
    \label{lem:s_hat_columnswitching}
    
\end{lem}

\begin{proof}
    As $\hat{S} = \frac{2}{n}J_{n} - I_{n}$, we can rewrite $\hat{S}\vec{x}$ as $\left(\frac{2}{n}J_{n} - I_{n}\right)\vec{x}$. We distribute $\vec{x}$ to get
    \begin{align*}
        \hat{S}\vec{x} = \frac{2}{n}J_{n}\vec{x} - I_{n}\vec{x}.
    \end{align*}
    As $\vec{x}$ is a vector with $\frac{n}{2}$ values of $a$ and $\frac{n}{2}$ values of $b$, then $J_{n}\vec{x}$ is the $n$-dimensional vector $\frac{n(a+b)}{2} \vec{1}$. Thus, we get
    \begin{align*}
        \hat{S}\vec{x} = \frac{2}{n} \left( \frac{n(a+b)}{2} \vec{1} \right) - \vec{x},
    \end{align*}
    which simplifies to $\hat{S}\vec{x} = (a+b)\vec{1} - \vec{x}$. As we define $\vec{y}=(a+b)\vec{1}-\vec{x}$, then $\hat{S}\vec{x} = \vec{y}$. Since $\hat{S}$ is symmetric, taking the transpose of each side yields $\vec{x}^T \hat{S} = \vec{y}^T$.
    
\end{proof}

We then show that multiplying $\hat{S}A\hat{S}$, where $A$ is a matrix with conditions stated in Lemma \ref{lem:simplifyshatmult}, returns $A$. We use this idea in our distance switching construction proof. 

\begin{lem}
\label{lem:simplifyshatmult} 
Let $A$ be an $n \times n$ matrix such that it has constant row and column sums of $m$. Let $\hat{S} = \frac{2}{n}J_{n} - I_{n}$. Then  $\hat{S}A\hat{S}=A$. 
\end{lem}

\begin{proof}

Consider $\hat{S}A\hat{S}$. Then, 
    
\begin{align*}
    \left(\frac{2J}{n} - I\right)(A)\left(\frac{2J}{n} - I\right) = \frac{4}{n^2}JAJ - \frac{2}{n}AJ - \frac{2}{n}JA + A 
\end{align*}
    Since the row and column sums of $A$ are $m$, if follows $JA = AJ=mJ$. Thus, we get the following:
\begin{align*}
    \frac{4}{n^2}JAJ - \frac{2}{n}AJ - \frac{2}{n}JA + A  &= \frac{4}{n}\left(\frac{JAJ}{n}- AJ\right) + A\\
    &= \frac{4}{n} \left( \frac{m}{n}J^2-mJ\right)+A\\
    &= \frac{4}{n} \left( \frac{mn}{n}J-mJ\right)+A\\
    &=A
\end{align*}
as desired.
\end{proof}

We will now use the Pigeonhole Principle to prove a fact about the number of neighbors for a vertex in a $\frac n2$-regular graph $G$, which we will later use in our construction. 

\begin{prop}\label{prop: n/2 pigeon hold}
    Let $G$ be an $d$-regular graph with $d \geq \frac{n}{2}$. Then for every subset $X$ of size at least $\frac{n}{2}$, every vertex in $V(G) \backslash X$ has at least one neighbor in $X$.
\end{prop}

\begin{proof}
    Let $G$ be an $d$-regular graph with $d\geq \frac{n}{2}$ and $X$ a subset of size $\frac{n}{2}$. Suppose there is a vertex $y \in V(G) \backslash X$ such it has no neighbor in $X$. Since $G$ is regular, it follows that $\deg(y) \geq \frac{n}{2}$, and $y$ is not adjacent to itself. So the $N(y) \subseteq V(G) \backslash (X \cup y)$, but $|V(G) \backslash (X \cup y)|\leq \frac{n}{2}-1$ and $|N(y)|\geq \frac{n}{2}$. Therefore, $y$ must have at least one neighbor in $X$.
\end{proof}

Using the previous Lemmas and Propositions, we propose the following construction for distance cospectral pairs. 

\begin{thm}
Let $A$ be a connected graph with partition $A^1, A^2, \ldots A^m$ where the graph induced by $A^i$ is at least $\frac{|A^i|}{2}$ regular and $|A^i|$ is even. Each vertex $a \in A^i$ must be adjacent to all or none of the vertices in each other $A^j$. Let $B^i$ be a collection of connected graphs $B_1^i, B_2^i, \ldots, B_{\ell}^i$ for $1 \leq i \leq m$. For each $B_j^i$, select half of the vertices in $A^i$ denoted $A_j^i$.

Let $G_1$ be the connected graph such that each vertex $b \in B_j^i$ is adjacent to none, all, or half of the vertices in $A^i$, where if $b$ is adjacent to half of the vertices, it is adjacent to $A_j^i$.  

Let $G_2$ be the connected graph $G_1$ where if $b$ is adjacent to exactly $A_j^i$ (half the vertices of $A^i$) in $G_1$, it will be adjacent to $A^i \backslash A_j^i$ in $G_2$.

Then $G_1$ and $G_2$ are cospectral for the distance matrix.
\label{thm:newgeneralizedconstruction}
\end{thm}

\begin{proof}
Define $G_1$ and $G_2$ as stated in Theorem \ref{thm:newgeneralizedconstruction}.
Let $S$ be the matrix

\begin{align*}
     S=
     I_{|B|}\oplus\hat{S}_{|A^1|}\oplus\hat{S}_{|A^2|}\oplus \dots \oplus\hat{S}_{|A^m|}
\end{align*}
where $\hat{S}$ is the same as defined in Lemma \ref{lem:simplifyshatmult}. We will show that $S$ is a similarity matrix for $D^{G_1}$ and $D^{G_2}$.

We can partition the distance matrix of $G_1$ as 

\begin{align*}
   \D^{G_1}= \left[ \begin{array}{c|c}
      \D^B & X \\
      \midrule
      X^T & \D^A \\  
\end{array}\right]\\
\end{align*}
with the first $|B|$ rows and columns corresponding to the vertices in each $B^i_j$ and the remaining rows and columns corresponding to the vertices in each $A^i$. Specifically, the submatrix $\D^A$ is structured as follows:

\begin{align*}
    \D^A=\left[\begin{array}{cccc}
        \mathcal{A}^{1} & \mathcal{A}^{1,2} & \dots & \mathcal{A}^{1,m} \\
        \mathcal{A}^{2,1} & \mathcal{A}^2 & \dots & \mathcal{A}^{2,m}\\
        \vdots & \vdots &\ddots & \vdots\\
        \mathcal{A}^{m,1} & \mathcal{A}^{2,m} & \dots & \mathcal{A}^m 
    \end{array}\right]
\end{align*}
Each $\mathcal{A}^i$ block contains the distances between the vertices in each partition of $A$, and each $\mathcal{A}^{i,k}$ contains the distances between vertices in $A^i$ and $A^k$. By our construction, either each vertex $a_i \in A^i$ is adjacent to each $a_k \in A^{k}$ or none are adjacent. In the first case, then $\dist(a_i, a_k) = 1$. If none are adjacent, as each $a_i \in A^i$ has the same neighborhood induced in the subgraph of $A$, any shortest path between every $a_i$ and every $a_k$ will have the same distance. Additionally, this shortest path will be contained in $A$. 

The submatrix $X$ is constructed as follows
\begin{align*}
    X=\left[\begin{array}{cccc}
        X^1 & X^2 & \dots & X^m \\
    \end{array}\right].
\end{align*}

Each row of $X$ represents the distances between a vertex $b$ in each connected graph $B^i_j$ to the vertices in $A$. In particular, each $X^i$ corresponds to the distances from $B$ to the vertices in $A^i$. We claim that for every $b \in B^i_j$ and $a \in A^k$, there are at most two possible distances between $b$ and $a$. In other words, for each row of $X^i$, there are at most two possible distances. If there are two distinct distances, then these distances to $b \in B^i_j$ appear an equal number of times.

We will examine the distances in $X$ by considering distances between the following cases: (1) $B_j^i$ to $A^i$ and (2) $B_j^i$ to $A^k$: 
\begin{enumerate}[label=(\arabic*)]
    \item $B_j^i$ to $A^i$: Let $B_j^i$ and $A^i$ be arbitrary and let $b \in B_j^i$ be adjacent to all or half of the vertices in $A^i$. If $b$ is adjacent to half the vertices in $A^i$, then half of the vertices in $A^i$ are distance 1 from $b$. Furthermore, by Proposition \ref{prop: n/2 pigeon hold}, all other vertices of $A^i$ are distance 2 from $b$. 
    If $b$ is adjacent to all of the vertices in $A^i$, then the distance from $b$ to any $a \in A^i$ is 1.

    Now assume $b \in B_j^i$ is not adjacent to any vertices in $A^i$. Let $b'$ denote the vertex in $B_j^i$ that is adjacent to at least one vertex in $A^i$ such that $\dist(b, b')$ is minimized. Thus, for any $a \in A^i$, $\dist(b, a) = \dist(b, b') + \dist(b', a)$. If $b'$ is adjacent to half of the vertices in $A^i$, as $\dist(b,b')$ does not depend on $a\in A^i$, for any $b \in B_j^i$ not adjacent to any vertex in $A^i$, there are two distances: $\dist(b,b') + 1$ and $\dist(b,b') + 2$, where half are the former and half are the latter. If $b'$ is adjacent to all of the vertices in $A^i$, then $\dist(b,a) = \dist(b,b') + 1$ for all $a \in A^i$.

    \item $B_j^i$ to $A^k$: First, fix $b\in B_j^i$. Then, the distance to all $a_k \in A^k$ is only going to have one value because there exists a shortest path from $b$ to the closest vertex in $A^i$. Since we established the distance from any vertex in $A^i$ to any vertex in $A^k$, we know the distance from $b$ to $a_k$ is this shortest distance from $b$ to the closest vertex in $A^i$ plus the distance from any vertex in $A^i$ to any vertex in $A^k$.

\end{enumerate}

Therefore, each row of $X^i$ has at most two possible values, representing the distances between vertices in each $B_j^i$ and each $A^i$ where if there are two values, they appear an equal number of times.

We now claim that in $G_2$, all the distances between $b\in B^i_j$ and $a\in A^i$ are switched from those in $G_1$. That is, $\tilde{X}$ is the matrix $X$ where the values of the rows of each $X^i$ are switched. 

Consider our proposed construction of $G_2$. For any $b \in B^i_j$ that is adjacent to half of the vertices in $A^i$ for the graph $G_1$, denote the half of $A^i$ that is adjacent to $b$ as $A^i_j$. Then, in $G_2$, we switch these edges so each $b$ is now adjacent to $A^i \backslash A^i_j$. Now, $b$ is distance 2 away from $a \in A^i_j$ by Proposition \ref{prop: n/2 pigeon hold}.

Consider the case where $b\in B_j^i$ is not adjacent to any $a \in A^i$ in $G_1$. Vertex $b$ will also not be adjacent to any $a \in A^i$ in $G_2$. As established, there were previously only two possible distances between any $b \in B_j^i$ and any $a \in A^i$: the vertices of $B_j^i$ that were previously distance $c$ to vertices of $A_j^i$ were also distance $d$ to the vertices of $A^i \backslash A^i_j$. In $G_2$, after the switching, the vertices of $B_j^i$ are now distance $d$ to the vertices of $A_j^i$ and distance $c$ to the vertices of $A^i \backslash A^i_j$. 

The distance between $a_i \in A^i$ and $a_j \in A^j$ will stay the same, as we are only switching the distances between each $B_j^i$ and $A^i$, which does not affect the distances between $a_i$ and $a_j$. This implies that the distance from any $B_j^i$ to $A^k$ where $i\neq k$ stays the same in $G_2$, as the distance from any $b \in B_j^i$ to the closest $a \in A^i$ is the same and the distances from each $A^i$ to $A^j$ also stay the same. It follows that the distance from any $B_j^i$ to $B_l^k$ will also remain the same.

Therefore, the distance matrix of $G_2$ is 
\begin{align*}
   \D^{G_2}=\left[ \begin{array}{c|c}
      \D^{B}  &  \tilde{X}\\ 
      \midrule
      \tilde{X}^{T}  & \D^{A} \\
   \end{array} \right]
\end{align*}
where  $\tilde{X}$ 
represents the distances between each $A^i \in A$ and each $B^i_j \in B$. 

Consider

\begin{align*}
   S \left[ \begin{array}{c|c}
      \D^{B}  & X \\
      \midrule
      X^T  & \D^{A} \\
   \end{array} \right] S &=\left[\begin{array}{ccccc}
        \D^B & X^1\hat{S} & X^2\hat{S} & \dots & X^m\hat{S} \\
        \hat{S}(X^1)^T & \hat{S}\mathcal{A}^1\hat{S} & \hat{S}\mathcal{A}^{1,2}\hat{S} & \dots & \hat{S}\mathcal{A}^{1,m}\hat{S} \\
        \hat{S}(X^2)^T & \hat{S}\mathcal{A}^{2,1}\hat{S} & \hat{S}\mathcal{A}^2\hat{S} & \dots & \hat{S}\mathcal{A}^{2,m}\hat{S}\\
        \vdots &\vdots & \vdots &\ddots & \vdots\\
        \hat{S}(X^m)^T & \hat{S}\mathcal{A}^{m,1}\hat{S} & \hat{S}\mathcal{A}^{2,m}\hat{S} & \dots & \hat{S}\mathcal{A}^m\hat{S}
    \end{array}\right]
\end{align*} 
Since each row of $X^i$ has at most two possible values, by Lemma \ref{lem:s_hat_columnswitching}, multiplying $X^i\hat{S}$ results in the switching of the values within a row of $X^i$. For example, if a row of $X^i$ contains $b$ and $a$, the $a$'s will switch to $b$'s and the $b$'s switch to $a$'s. This will result in $\tilde{X}$ by construction.
So we have shown,  \begin{align*}
\tilde{X} =
    \left[\begin{array}{cccc}
        X^1\hat{S}  & X^2\hat{S}  & \dots & X^m\hat{S} 
    \end{array}\right].
\end{align*}

Similarly, \begin{align*}
\tilde{X}^T =
    \left[\begin{array}{cccc}
        \hat{S}(X^1)^T  & \hat{S}(X^2)^T  & \dots & \hat{S}(X^m)^T
    \end{array}\right]^T.
\end{align*}

We will now examine the result of the matrix multiplication on the block $\D^A$. Consider the block $\mathcal{A}^i$. Since $A^i$ is regular, we know that every row and column of $\mathcal{A}^i$ will have the same number of ones. Then, Proposition \ref{prop: n/2 pigeon hold} tells us the rest of the off-diagonal entries will be two. Thus, $\mathcal{A}^i$ has constant row and column sums. So, by Lemma \ref{lem:simplifyshatmult}, we know $\hat{S}\mathcal{A}^i\hat{S} = \mathcal{A}^i$. Now consider the block $\mathcal{A}^{i,k}$. By construction, each off-diagonal entry of $\mathcal{A}^{i,k}$ is the same. Thus, $\mathcal{A}^{i,k}$ also has constant row and column sums, and Lemma \ref{lem:simplifyshatmult} tells us $\hat{S}\mathcal{A}^{i,k}\hat{S} = \mathcal{A}^{i,k}$. We have now shown that
\begin{align*}
    \left[\begin{array}{cccc}
        \hat{S}\mathcal{A}^1\hat{S} & \hat{S}\mathcal{A}^{1,2}\hat{S} & \dots & \hat{S}\mathcal{A}^{1,m}\hat{S} \\
        \hat{S}\mathcal{A}^{2,1}\hat{S} & \hat{S}\mathcal{A}^2\hat{S} & \dots & \hat{S}\mathcal{A}^{2,m}\hat{S}\\
        \vdots & \vdots &\ddots & \vdots\\
        \hat{S}\mathcal{A}^{m,1}\hat{S} & \hat{S}\mathcal{A}^{2,m}\hat{S} & \dots & \hat{S}\mathcal{A}^m\hat{S}
    \end{array}\right] =\left[\begin{array}{cccc}
        \mathcal{A}^1 & \mathcal{A}^{1,2} & \dots & \mathcal{A}^{1,m} \\
        \mathcal{A}^{2,1} & \mathcal{A}^2 & \dots & \mathcal{A}^{2,m}\\
        \vdots & \vdots &\ddots & \vdots\\
        \mathcal{A}^{m,1} & \mathcal{A}^{2,m} & \dots & \mathcal{A}^m 
    \end{array}\right] = \D^A.
\end{align*}

\medspace{}

Therefore, 
\begin{align*}
    S \left[ \begin{array}{c|c}
      \D^{B}  & X \\
      \midrule
      X^T  & \D^{A} \\
   \end{array} \right] S=\left[ \begin{array}{c|c}
      \D^{B}  &  \tilde{X}\\ 
      \midrule
      \tilde{X}^{T}  & \D^{A} \\
   \end{array} \right].
\end{align*}

So, $S\D^{G_1}S^{-1} = \D^{G_2}$, and $S$ is a similarity matrix for cospectral pair $G_1$ and $G_2$. 
\end{proof}

We will now add several extensions to our construction. We show that we can add edges between vertices $b \in B_i$ and we extend the construction to the generalized distance matrix. We can also apply coalescing from \cite{coalescing} to our construction.

\subsection{Extensions of Main Construction}


Our first extension allows for edges to be added between vertices. This was not included in the initial construction for ease of proof. 

\begin{cor}
\label{cor:addadjacencies}
    For every pair of vertices $b_1\in B_j^i$ and $b_2 \in B_k^i$, where both $b_1,b_2$ are adjacent to at least one vertex in $A^i$. Then $G_1+\{b_1,b_2\}$ is cospectral to $G_2+\{b_1,b_2\}$. 
\end{cor}

\begin{proof}
    Adding edges between pairs of vertices $b \in B^i$ that are both neighbors to a vertex $a \in A^i$ will keep the distances to $A$ the same, as both $b$ are already joined to $A^i$. Therefore, the algebraic process used in the proof of Theorem \ref{thm:newgeneralizedconstruction} applies. 
\end{proof}

In addition to adding edges between $b \in B_i$, we can also extend our construction to the generalized distance matrix. Since the functionality of the matrix switching in \ref{thm:newgeneralizedconstruction} does not rely on the actual values of the distance function, we are able to perform the same operations using the generalized distance matrix. Recall that the generalized distance matrix, $\Df$, is the matrix where $(\Df)_{i,j} = f(\dist(i,j))$ for any function $f$. So, our construction applies to any function of distance. 

\begin{cor}
\label{arbdistswitchingcor}
    $G_1$ and $G_2$ are also cospectral for the generalized distance matrix.
\end{cor}

\begin{proof}
    We use the same process as used for Theorem \ref{thm:newgeneralizedconstruction}, 
    but we replace each $\dist(b,a)$ with $f(\dist(b,a))$ for an arbitrary function $f$. Even with this change, the matrix algebra of the proof remains the same. Thus, we are still switching the entries in blocks $X$ and $X^T$. Therefore, $D_f^{G_1}$ and $D_f^{G_2}$ are similar for the same $S$, so $G_1$ and $G_2$ are cospectral for the generalized distance matrix.
\end{proof}

In expanding our construction to the generalized distance matrix, we significantly expand its scope. In Section \ref{sec:data}, we present the number of cospectral graphs, on small numbers of vertices, generated by our construction. However, we can also apply coalescing from Bin Mahamud et. al in \cite{coalescing} to the construction. The authors define \emph{coalescing}  as the following: 
\begin{defn}[\cite{coalescing}]
    Given graphs $G$ with vertex $v$ and a rooted graph $H$, the \emph{coalescing of $H$ onto $v$}, denoted $G\coal_vH$, is the graph formed by taking the disjoint union of the graphs $G$ and $H$ and then identifying the root of $H$ and $v$ as the same vertex. 
\end{defn}

Now, using a result from Bin Mahamud et. al in \cite{coalescing}, we show that we can coalesce connected graphs $C$ onto vertices of $A$.

\begin{cor}
Let $C$ be a connected graph and choose a vertex of $C$, denoted $c$. Let $H_1$ and $H_2$ be the graphs with copies of $C$ coalesced to all vertices of $A^i$ for some $1 \leq i \leq m$ of $G_1$ and $G_2$ respectively at vertex $c$. Then $H_1$ and $H_2$ are cospectral for the generalized distance matrix.
\label{cor:coalescing}
\end{cor}

\begin{proof}
    Since Theorem \ref{thm:newgeneralizedconstruction} holds for the generalized distance function with the similarity matrix $S$, it follows that we can coalesce copies of $C$ onto every vertex in $A^i$ as stated in Theorem 3.3 of \cite{coalescing}.
\end{proof}

\subsection{Examples}

In this section, we present illustrative examples of cospectral pairs to highlight our main construction and its corollaries.

Our first example highlights Theorem \ref{thm:newgeneralizedconstruction} shown in Figure \ref{fig:small-example}. Let $A^1=\{4,5,6,7\}$ and $B_1^1=\{1\}, B_2^1=\{2\}, B_3^1=\{3\}$. Notice that the vertices of $A$ induce a $2$-regular graph and each $B_i^1$ is adjacent to half the vertices of $A^1$. Further, by switching the adjacencies in $G_1$, (Figure \ref{fig:small-example}a) we obtain $G_2$ (Figure \ref{fig:small-example}b). These graphs are cospectral for the generalized distance matrix and can use the similarity matrix $S$ that we describe in Theorem \ref{thm:newgeneralizedconstruction}. 
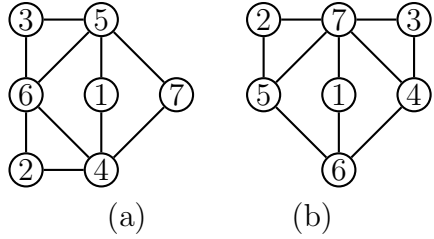
\begin{figure}[h]
    \centering
   \begin{tikzpicture}[scale = 1]

\node[Vertex] (a) at ({2},1) {$7$};
\node[Vertex] (b) at ({0},0) {$2$};
\node[Vertex] (c) at ({0},1) {$6$};
\node[Vertex] (d) at (1,0) {$4$};
\node[Vertex] (e) at (1,2) {$5$};
\node[Vertex] (f) at (1,1) {$1$};
\node[Vertex] (g) at ({0},2) {$3$};

\draw[Edge] (c)--(d)--(f)--(e)--(c)--(b)--(d)--(a);
\draw[Edge] (c)--(g)--(e);
\draw[Edge] (a)--(e);

\end{tikzpicture} \hspace{1em} 
\begin{tikzpicture}[scale = 1]

\node[Vertex] (a) at (1,1) {$1$};
\node[Vertex] (b) at ({2},2) {$3$};
\node[Vertex] (c) at ({2},1) {$4$};
\node[Vertex] (d) at (1,2) {$7$};
\node[Vertex] (e) at (1,0) {$6$};
\node[Vertex] (f) at ({0},1) {$5$};
\node[Vertex] (g) at ({0},2) {$2$};

\draw[Edge] (c)--(d)--(f)--(e)--(c)--(b)--(d)--(a) (d)--(g)--(f)  (a)--(e);

\end{tikzpicture} \\
(a) \hspace{4em} (b)
    \caption{Two graphs (a) $G_1$ and (b) $G_2$ that are cospectral for the generalized distance matrix. Let $A^1=\{4,5,6,7\}$ and $B_1^1=\{1\}, B_2^1=\{2\}, B_3^1=\{3\}$. By switching the adjacencies between $B_i^1$ and $A^1$ we obtain $G_2$ from $G_1$.}
    \label{fig:small-example}
\end{figure}

Another example of our construction, with multiple connected regular graphs $A^1, A^2,$ and $A^3$, is shown in Figure \ref{fig:cospecmultipleAi}. Both of these graphs have diameter $5$. Adjacencies between each $A^i$ are shown in gray. Every $a \in A^1$ is adjacent to every vertex in $A^2$ and no vertex in $A^3$. Every $a \in A^3$ is adjacent to every vertex in $A^2$ and no vertex in $A^1$.

The adjacencies from $b \in B^i$ to half of the vertices of the corresponding $A^i$ are indicated by colored lines. With the addition of the edges between vertices in $B_i^2$, this partition satisfies the conditions of Corollary \ref{cor:addadjacencies}. These adjacencies in $G_1$ (Figure \ref{fig:cospecmultipleAi}a) switch to the other half of the vertices in each $A^i$ in $G_2$ (Figure \ref{fig:cospecmultipleAi}b).

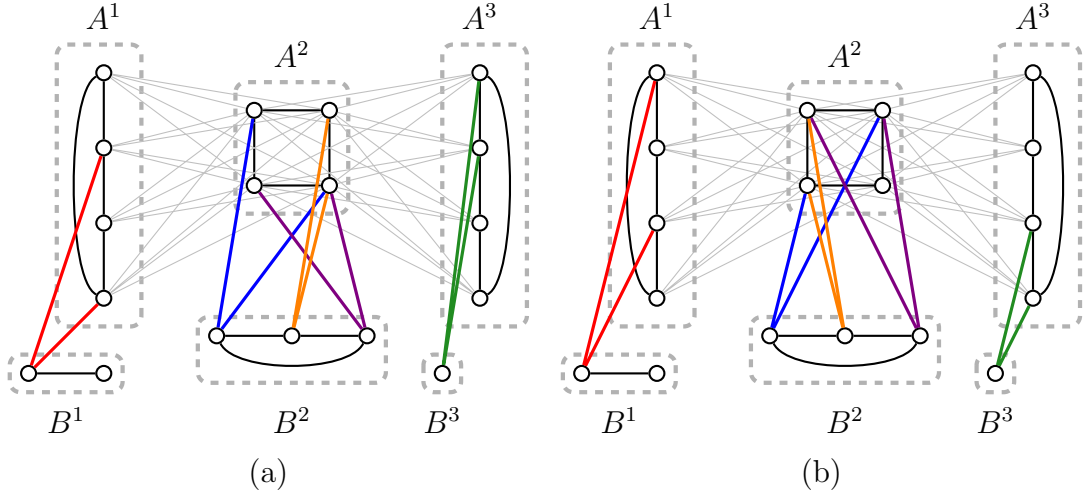
\begin{figure}[h]
    \centering
    \begin{tabular}{cc}
       \begin{tikzpicture}[scale=0.5]
           
            \draw[rounded corners, color=black!30!white,ultra thick, dashed] (4,-3.75) rectangle (6.25,3.75)
            (-1.5, 2.75) rectangle (1.5,-0.75)
            
            (-4, -3.75) rectangle (-6.25, 3.75) (-4.5, -4.5) rectangle (-7.5, -5.5)
            (-2.5, -3.5) rectangle (2.5,-5.25)  
             (3.5, -4.5) rectangle (4.5, -5.5);
            
             \node (a1) at (-5, 4.5) {$A^1$};
             \node (a2) at (0, 3.5) {$A^2$};
             \node (a3) at (5, 4.5) {$A^3$};
             \node (b1) at (-6, -6.25) {$B^1$};
             \node (b2) at (0, -6.25) {$B^2$};
             \node (b3) at (4, -6.25) {$B^3$};

            \node[vertex] (0) at (5,3) {};
            \node[vertex] (1) at (5,1) {};
            \node[vertex] (2) at (5,-1) {};
            \node[vertex] (3) at (5,-3) {};

            \node[vertex] (4) at (-1,2) {};
            \node[vertex] (5) at (-1, 0) {};
            \node[vertex] (6) at (1, 0) {};
            \node[vertex] (7) at (1, 2) {};

            \node[vertex] (8) at (-5,3) {};
            \node[vertex] (9) at (-5,1) {};
            \node[vertex] (10) at (-5,-1) {};
            \node[vertex] (11) at (-5,-3) {};
            
            \draw[thick] (3)--(2)--(1)--(0);
            \draw[thick] (3) .. controls (6, -2.5) and (6,2.5) .. (0);
            \draw[thick] (4)--(5)--(6)--(7)--(4);
            \draw[thick] (8)--(9)--(10)--(11);
           \draw[thick] (11) .. controls (-6, -2.5) and (-6, 2.5) .. (8);

           \node[vertex] (12) at (-5, -5) {};
           \node[vertex] (13) at (-7, -5) {};
           \draw[thick] (12)--(13);

           
           \node[vertex] (15) at (0, -4) {};
           \node[vertex] (14) at (-2, -4) {};
           \node[vertex] (16) at (2, -4) {};
           \draw[thick] (14)--(15)--(16);
           \draw[thick] (14) .. controls (-1.5, -5) and (1.5, -5) .. (16);

           \node[vertex] (17) at (4, -5) {};

           \draw[lightgray] (0)--(4) (0)--(5) (0)--(6) (0)--(7);
           \draw[lightgray] (1)--(4) (1)--(5) (1)--(6) (1)--(7);
           \draw[lightgray] (2)--(4) (2)--(5) (2)--(6) (2)--(7);
           \draw[lightgray] (3)--(4) (3)--(5) (3)--(6) (3)--(7);

           \draw[lightgray] (8)--(4) (8)--(5) (8)--(6) (8)--(7);
           \draw[lightgray] (9)--(4) (9)--(5) (9)--(6) (9)--(7);
           \draw[lightgray] (10)--(4) (10)--(5) (10)--(6) (10)--(7);
           \draw[lightgray] (11)--(4) (11)--(5) (11)--(6) (11)--(7);
           
           \draw[thick, red, line width=1.2pt] (13)--(11) (13)--(9);

           \draw[thick, forestgreen, line width=1.2pt] (17)--(0) (17)--(1);

           \draw[thick, blue, line width = 1.2pt] (14)--(4) (14)--(6);
           \draw[thick, violet, line width=1.2 pt] (16)--(5) (16)--(6); 
           \draw[thick, orange, line width=1.2pt] (15)--(6) (15)--(7);       
           
        \end{tikzpicture}  &      
        \begin{tikzpicture}[scale=0.5]
           \draw[lightgray] (0)--(4) (0)--(5) (0)--(6) (0)--(7);
           \draw[lightgray] (1)--(4) (1)--(5) (1)--(6) (1)--(7);
           \draw[lightgray] (2)--(4) (2)--(5) (2)--(6) (2)--(7);
           \draw[lightgray] (3)--(4) (3)--(5) (3)--(6) (3)--(7);

           \draw[lightgray] (8)--(4) (8)--(5) (8)--(6) (8)--(7);
           \draw[lightgray] (9)--(4) (9)--(5) (9)--(6) (9)--(7);
           \draw[lightgray] (10)--(4) (10)--(5) (10)--(6) (10)--(7);
           \draw[lightgray] (11)--(4) (11)--(5) (11)--(6) (11)--(7);
            \draw[rounded corners, color=black!30!white,ultra thick, dashed] (4,-3.75) rectangle (6.25,3.75)
            (-1.5, 2.75) rectangle (1.5,-0.75)
            (-2.5, -3.5) rectangle (2.5,-5.25)  
             (3.5, -4.5) rectangle (4.5, -5.5)
            
            (-4, -3.75) rectangle (-6.25, 3.75) (-4.5, -4.5) rectangle (-7.5, -5.5);
             \node (a1) at (-5, 4.5) {$A^1$};
             \node (a2) at (0, 3.5) {$A^2$};
             \node (a3) at (5, 4.5) {$A^3$};
             \node (b1) at (-6, -6.25) {$B^1$};
             \node (b2) at (0, -6.25) {$B^2$};
             \node (b3) at (4, -6.25) {$B^3$};

            \node[vertex] (0) at (5,3) {};
            \node[vertex] (1) at (5,1) {};
            \node[vertex] (2) at (5,-1) {};
            \node[vertex] (3) at (5,-3) {};

            \node[vertex] (4) at (-1,2) {};
            \node[vertex] (5) at (-1, 0) {};
            \node[vertex] (6) at (1, 0) {};
            \node[vertex] (7) at (1, 2) {};

            \node[vertex] (8) at (-5,3) {};
            \node[vertex] (9) at (-5,1) {};
            \node[vertex] (10) at (-5,-1) {};
            \node[vertex] (11) at (-5,-3) {};
            
            \draw[thick] (3)--(2)--(1)--(0);
            \draw[thick] (3) .. controls (6, -2.5) and (6,2.5) .. (0);
            \draw[thick] (4)--(5)--(6)--(7)--(4);
            \draw[thick] (8)--(9)--(10)--(11);
           \draw[thick] (11) .. controls (-6, -2.5) and (-6, 2.5) .. (8);

           \node[vertex] (12) at (-5, -5) {};
           \node[vertex] (13) at (-7, -5) {};
           \draw[thick] (12)--(13);

           
           \node[vertex] (15) at (0, -4) {};
           \node[vertex] (14) at (-2, -4) {};
           \node[vertex] (16) at (2, -4) {};
           \draw[thick] (14)--(15)--(16);
           \draw[thick] (14) .. controls (-1.5, -5) and (1.5, -5) .. (16);

           \node[vertex] (17) at (4, -5) {};


           \draw[thick, red, line width=1.2pt] (13)--(10) (13)--(8);

           \draw[thick, forestgreen, line width=1.2pt] (17)--(2) (17)--(3);

           \draw[thick, blue, line width = 1.2pt] (14)--(5) (14)--(7);
           \draw[thick, violet, line width=1.2 pt] (16)--(4) (16)--(7); 
           \draw[thick, orange, line width=1.2pt] (15)--(4) (15)--(5);

        \end{tikzpicture}\\
        (a) & (b) \end{tabular}
      
    \caption{Two graphs (a) $G_1$ and (b) $G_2$ that are cospectral for the generalized distance matrix with $\diam(G_1)=\diam(G_2)=5$. The vertex partition noted on the graph satisfies the partition in Theorem \ref{thm:newgeneralizedconstruction} and Corollary \ref{cor:addadjacencies} with the edges added between vertices from $B_i^2$.}
    \label{fig:cospecmultipleAi}
\end{figure}

\begin{figure}[h]
    \centering
    \setlength{\tabcolsep}{24pt}
    \begin{tabular}{cc}
        \begin{tikzpicture}[scale=0.5]
            \draw[rounded corners, color=black!30!white,ultra thick, dashed] (2.5,-3.75) rectangle (4.25,3.75)
            (-4.75, .25) rectangle (.75,3.75)
            (-2.75, -4.75) rectangle (.75,-3.25)
            (-1.75, -2.25) rectangle (-.25, -.75);

            \node (f1) at (-5.5, 2) {$B^1$};
            \node (f3) at (-2.5, -1.5) {$B^2$};
            \node (f2) at (-3.5, -4) {$B^3$};
            \node (a) at (3.25, -4.25) {$A$};
            
            \node[vertex] (0) at (3,3) {};
            \node[vertex] (1) at (3,1) {};
            \node[vertex] (2) at (3,-1) {};
            \node[vertex] (3) at (3,-3) {};

            \node[vertex] (4) at (0, 2) {};
            \node[vertex] (5) at (0, -4) {};

            \node[vertex] (6) at (-2, 2) {};
            \node[vertex] (7) at (-2, -4) {};
            \node[vertex] (18) at (-1, -1.5) {};

            \node[vertex] (8) at (-4, 3) {};
            \node[vertex] (9) at (-4, 1) {};

            \node[vertex] (10) at (5, 3.5) {};
            \node[vertex] (11) at (5,2.5) {};
            \node[vertex] (12) at (5, 1.5) {};
            \node[vertex] (13) at (5,0.5) {};
            \node[vertex] (14) at (5, -0.5) {};
            \node[vertex] (15) at (5,-1.5) {};
            \node[vertex] (16) at (5, -2.5) {};
            \node[vertex] (17) at (5,-3.5) {};

            \draw[thick] (0)--(1)--(2)--(3);
            \draw[thick] (3) .. controls (4, -2.5) and (4,2.5) .. (0);
            \draw[thick, blue, line width=1.2pt] (6)--(0) (6)--(2);
            \draw[thick] (4)--(0) (4)--(1) (4)--(2) (4)--(3);
            \draw[thick] (6)--(8)--(9)--(6)--(4);
            \draw[thick, red, line width=1.2pt] (5)--(1) (5)--(2) (7)--(1) (7)--(2);
            \draw[thick] (5)--(7);
            \draw[thick] (10)--(11)--(0)--(10) (12)--(13)--(1)--(12) (14)--(15)--(2)--(14) (16)--(17)--(3)--(16);
            \draw[thick, forestgreen, line width=1.2pt] (0)--(18)--(1);
        \end{tikzpicture}
        & \begin{tikzpicture}[scale=0.5]
            \draw[rounded corners, color=black!30!white,ultra thick, dashed] (2.5,-3.75) rectangle (4.25,3.75)
            (-4.75, .25) rectangle (.75,3.75)
            (-2.75, -4.75) rectangle (.75,-3.25)
            (-1.75, -2.25) rectangle (-.25, -.75);

            \node (f1) at (-5.5, 2) {$B^1$};
            \node (f3) at (-2.5, -1.5) {$B^2$};
            \node (f2) at (-3.5, -4) {$B^3$};
            \node (a) at (3.25, -4.25) {$A$};
            
            \node[vertex] (0) at (3,3) {};
            \node[vertex] (1) at (3,1) {};
            \node[vertex] (2) at (3,-1) {};
            \node[vertex] (3) at (3,-3) {};

            \node[vertex] (4) at (0, 2) {};
            \node[vertex] (5) at (0, -4) {};

            \node[vertex] (6) at (-2, 2) {};
            \node[vertex] (7) at (-2, -4) {};
            \node[vertex] (18) at (-1, -1.5) {};

            \node[vertex] (8) at (-4, 3) {};
            \node[vertex] (9) at (-4, 1) {};

            \node[vertex] (10) at (5, 3.5) {};
            \node[vertex] (11) at (5,2.5) {};
            \node[vertex] (12) at (5, 1.5) {};
            \node[vertex] (13) at (5,0.5) {};
            \node[vertex] (14) at (5, -0.5) {};
            \node[vertex] (15) at (5,-1.5) {};
            \node[vertex] (16) at (5, -2.5) {};
            \node[vertex] (17) at (5,-3.5) {};

            \draw[thick] (0)--(1)--(2)--(3);
            \draw[thick] (3) .. controls (4, -2.5) and (4,2.5) .. (0);
            \draw[thick, blue, line width=1.2pt] (6)--(1) (6)--(3);
            \draw[thick] (4)--(0) (4)--(1) (4)--(2) (4)--(3);
            \draw[thick] (6)--(8)--(9)--(6)--(4);
            \draw[thick, red, line width=1.2pt] (5)--(0) (5)--(3) (7)--(0) (7)--(3);
            \draw[thick] (5)--(7);
            \draw[thick] (10)--(11)--(0)--(10) (12)--(13)--(1)--(12) (14)--(15)--(2)--(14) (16)--(17)--(3)--(16);
            \draw[thick, forestgreen, line width=1.2pt] (3)--(18)--(2);
        \end{tikzpicture} \\
        (a) & (b)
    \end{tabular}
    \caption{Two graphs (a) $G_1$ and (b) $G_2$. The partition denoted on the graph (with gray dotted lines) satisfies the conditions of Theorem \ref{thm:newgeneralizedconstruction}, and thus, the pair is cospectral for the generalized distance matrix. By Corollary \ref{cor:coalescing}, we coalescing $K_3$ to each vertex in $A$ to maintain cospectrality for the generalized distance matrix.}
    \label{fig:coalesce_construct_ex}
\end{figure}
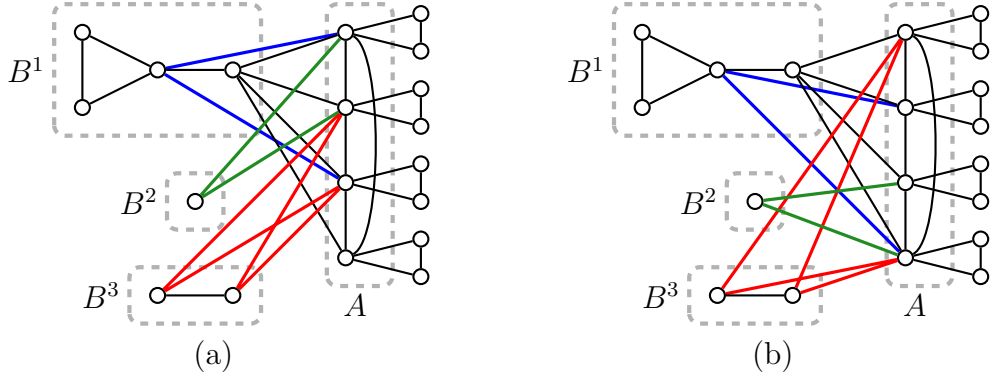

Corollary \ref{cor:coalescing} states that we can coalesce copies of $C$ onto both $G_1$ and $G_2$, retaining cospectrality. In Figure \ref{fig:coalesce_construct_ex}, we start with a $G_1$ and $G_2$ that satisfy the vertex partition of Theorem \ref{thm:newgeneralizedconstruction} and then coalesce a $C_3$ to each vertex of $A$. Therefore, these new graphs $H_1$ and $H_2$ are cospectral for the generalized distance matrix. 

Through these examples, we hope to demonstrate the breadth and scope of our construction method. In our next section, we provide some data on the prevalence of the method. 

\section{Data}
\label{sec:data}
    To evaluate the significance of our cospectral construction, we computed the total number of cospectral pairs on 7, 8, and 9 vertices for both the exponential and generalized distance matrices. Table \ref{tab:constructiondata} includes the total number of cospectral graph pairs for both matrices and identifies how many of these pairs arise from the constructions described in Theorem \ref{thm:newgeneralizedconstruction} and Corollary \ref{cor:addadjacencies}. These results demonstrate that our construction plays a substantial role in generating cospectral pairs for graphs of these sizes. See \cite{code} for code used to generate this data.

\begin{table}[ht]
\def\arraystretch{1.2}
\centering
\begin{tabular}{r|cc|cccc|}
\cline{2-7}
\multicolumn{1}{l|}{}                         & \multicolumn{2}{c|}{Total cospectral pairs} & \multicolumn{4}{c|}{Pairs following our construction}                       \\ \hline
\multicolumn{1}{|l|}{\textbf{\# of vertices}} & \textbf{$\Dq$}       & \textbf{$\Df$}       & \textbf{$\Dq$} & \textbf{$\Dq$ (\%)} & \textbf{$\Df$} & \textbf{$\Df$ (\%)} \\ \hline
\multicolumn{1}{|r|}{7} & 11    & 10    & 10  & 91\% &  10 & 100\% \\
\multicolumn{1}{|r|}{8} & 293   & 281   & 222 & 76\% &  222 & 79\%  \\
\multicolumn{1}{|r|}{9} & 12439 & 12118 & 6375 &   51\%   & 6375 &   53\%    \\ \hline
\end{tabular}
\caption{Total number of cospectral pairs and how many follow the construction presented by Theorem \ref{thm:newgeneralizedconstruction} and Corollary \ref{cor:addadjacencies}}
\label{tab:constructiondata}
\end{table}

\section{When Does Cospectrality Hold for All Values of $q$?}
\label{sec:cospecforallq}
Though we created a construction for the generalized distance matrix, the question remains: for which values of $q$ are two graphs cospectral? We found that the number of $q$ values needed to prove cospectrality for all $q$ depends on the maximum diameter of $G$ and $H$. Additionally, if a pair of graphs is cospectral for all $q$ for the same similarity matrix $S$, it will also be cospectral for the generalized distance matrix. The proof follows in Theorem \ref{thm:diamqallq}.

\begin{thm}
\label{thm:diamqallq}
Let $G$ and $H$ be graphs with connected components $G=G_1 \cup \cdots \cup G_r$ and $H=H_1 \cup \cdots \cup H_s$. Let $d= \max\{\diam(G_i), \diam(H_i)\}$ (the maximum diameter of the connected components). If $G,H$ are exponential distance cospectral for $d$ distinct non-zero values of $q$ for the same similarity matrix $S \in \mathbb{C}^{n \times n}$, then $G$ and $H$ are cospectral for the generalized distance matrix.
\end{thm}

\begin{proof}
First, we will show that a pair of graphs that is cospectral for $d$ distinct values of $q$ on the same similarity matrix will be cospectral for all values of $q$. Since $G$ and $H$ are cospectral for the exponential distance matrix with the similarity matrix $S$ for $d$ distinct non-zero values of $q$, then $\D_q^GS = S\D_q^H$ for $d$ values of $q$. Thus, $(\D_q^GS)_{i,j} = (S\D_q^H)_{i,j}$. We write $(\D_q^GS = \D_q^HS)_{i,j}$ as follows: 

   \begin{align*}
       x_0 + x_1q +x_2q^2 + \dots + x_dq^d = y_0 +y_1q +y_2q^2 + \dots + y_dq^d
   \end{align*}
   where $x_k$ is the sum of the entries in the $j$th column of $S$ corresponding to $q^k$ and $y_k$ is the sum of the entries in the $i$th row of $S$ corresponding to $q^k$. It's possible for $x_k$ and $y_k$ to be $0$ for some $q^k$ terms. We subtract the right-hand side from the equation to get
   
   \begin{align}
   \label{eq:exp_dist_mult_entry}
       (x_0-y_0)+(x_1-y_1)q + (x_2-y_2)q^2 + \dots + (x_d-y_d)q^d = 0.
   \end{align}
   Since we assumed $G$ and $H$ are cospectral for $d$ distinct non-zero values of $q$, these same values of $q$ must be roots for the polynomial above. Since the polynomial has $d$ non-zero roots in addition to a root of $q=0$, the polynomial must have $d+1$ roots. Therefore the polynomial has degree $d$ and $d+1$ roots so it must be the zero-polynomial by the Fundamental Theorem of Algebra. Therefore, the coefficients of each $q^k$ must be $0$, so Equation \ref{eq:exp_dist_mult_entry} holds for any value of $q$. So, $G$ and $H$ are cospectral for arbitrary $q$.

   Furthermore, since every coefficient of Equation \ref{eq:exp_dist_mult_entry} is zero, we can also replace $q^k$ with any function of $k$ such that
   \begin{align*}
       (x_0-y_0)f(0)+(x_1-y_1)f(1) + (x_2-y_2)f(2) + \dots + (x_d-y_d)f(d) = 0,
   \end{align*}
   which can be rewritten as
   \begin{align*}
       x_0f(0) + x_1f(1) +x_2f(2) + \dots + x_df(d) = y_0f(0) +y_1f(1) +y_2f(2) + \dots + y_df(d).
   \end{align*}
   
Therefore, $S\D_f^G = \D_f^HS$, so $G$ and $H$ are cospectral for the generalized distance matrix.
\end{proof}

Our sufficient condition for generalized distance cospectrality means that if we find that a graph of diameter $d$ is cospectral for $d$ values of $q$ for the same similarity matrix, then the graph is cospectral for the generalized distance matrix. Using this method streamlines searching for generalized distance cospectral graphs and requires a relatively small amount of checks, especially for low diameter graphs.

\section{Cospectral constructions for exponential distance matrix for unique value of $q$}
\label{sec:cospecconstructions}
We have seen cospectral graphs for matrices with functions of the distance, including the exponential distance matrix for all values of $q$. Cospectrality for all $q$ would indicate that the choice of $q$ is arbitrary. However, there are some values of $q$ that have interesting results, as noted by Butler et al. in \cite{expdist}. For example, when letting $q=\frac12$, there are some unique cospectral pairs of graphs. Here, we show a cospectral construction technique that describes two such families, shown in Figures \ref{fig:dq12-cospec-friendship} and \ref{fig:dq12-cospec-complete-star}. 

Our construction shows that the two graphs have the same characteristic polynomial. Recall that the characteristic polynomial of a matrix $M$ is $\mathcal{P}_M(x)=\det(M-xI)$, where 
\begin{align*}
    \det(A)=\sum_{\pi \in S_n}sgn(\pi)a_{1,\pi(1)}a_{2,\pi(2)}\cdots a_{n,\pi(n)}.
\end{align*}

We will also use a recursive formula for the characteristic polynomial of a path graph for the exponential distance matrix derived by Butler et al. in \cite{expdist}.

\begin{prop}[\cite{expdist}]\label{prop:charpoly-recursive}
Let $\mathcal{P}_{P_n}(x)$ be the characteristic polynomial of path graph for the exponential distance matrix on $n$ vertices. Then,\begin{align*}
    \mathcal{P}_{P_{n}}(x) = \big((q^2+1)x-1+q^2\big) \mathcal{P}_{P_{n-1}}(x) - x^2q^2 \mathcal{P}_{P_{n-2}}(x).
\end{align*}    
\end{prop}

We now state our construction. 

\begin{thm}
\label{box-windmill cospec}
    The graphs $G$ and $H$ in Figure \ref{fig:dq12-cospec-friendship} are cospectral on the exponential distance matrix when $q = \frac{1}{2}$. 
\end{thm}

\begin{proof}
In order to prove cospectrality between $G$ and $H$, we will show they have the same characteristic polynomial. 

 Since graph $H$ has two separate connected components, we find its characteristic polynomial is the characteristic of each component multiplied together. Therefore
\begin{align*}
    \mathcal{P}_{H}(x) = (2q^2 - q + x - 1)(q + x - 1)^3 \cdot (2q^3 - 2q^2x - 5q^2 - 2qx + x^2 + 2q - 2x + 1) \cdot \mathcal{P}_{P_{n-6}}(x).
\end{align*}
Now, we turn our attention to finding the characteristic polynomial of graph $G$ in terms of $\mathcal{P}_{P_{n-6}}(x)$. Let us partition the graph into three parts: vertices $1,2$ which are both adjacent to vertex $n$, vertices $3,4,5,6$ which are all adjacent to vertex $7$, and vertices $7,8,\ldots, n$, the induced subgraph of which forms a path on $n-6$ vertices. 
Therefore, 
\begin{align*}
    \Dq^G=\left[\begin{array}{ccc}
        \Dq^{G_{[1,2]}} & A & B \\
        A^T & \Dq^{G_{[3,4,5,6]}} & C\\
        B^T & C^T &\Dq^{P_{n-6}}
    \end{array}\right].
\end{align*}
Consider $\mathcal{P}_{\Dq^G}(x)=\det(\Dq^G-xI)$. So we are looking to find the determinant of 
\begin{align*}
    \left[\begin{array}{ccc}
        \Dq^{G_{[1,2]}}-xI & A & B \\
        A^T & \Dq^{G_{[3,4,5,6]}}-xI & C\\
        B^T & C^T &\Dq^{P_{n-6}}-xI
    \end{array}\right].
\end{align*}
Looking at the submatrix $\left[\begin{array}{c}
     B\\
     C
\end{array}\right] =  \left[\begin{array}{ccccc}
q^{n-6} & q^{n-7} & \dots & q^2 & q\\
        q^{n-6} & q^{n-7} & \dots & q^2 & q\\
        q & q^2 & \dots & q^{n-7} & q^{n-6}\\
         q & q^2 & \dots & q^{n-7} & q^{n-6}\\
          q & q^2 & \dots & q^{n-7} & q^{n-6}\\
         q & q^2 & \dots & q^{n-7} & q^{n-6}\\
\end{array}\right]$, we can see that many of the rows are similar to each other.  Further, rows $1$ and $2$ are $q$ times the last row of $\Dq^{P_{n-6}}$ and rows $3,4,5,6$ are $q$ times the first row of $\Dq^{P_{n-6}}$.  So we can add multiples of rows and columns to each other, which does not effect the determinant. After the row and column operations, we will have the resulting matrix 

\begin{align*}
\tiny{
    \left[\begin{array}{cc|cccc|ccccc}
         2q+2x-2 & 1-q-x & 0 & 0 & 0 & 0& 0 & 0 & \dots & 0 & 0\\
    1-q-x & q^2x+q^2+x-1 & 0 & 0 & 0 & 0&0 & 0 & \dots & 0 & -qx\\
    \hline
    0 & 0 & 2q+2x-2 & 1-q-x & 0 & 0 & 0 & 0 & \dots & 0 & 0\\
    0 & 0 & 1-q-x & 2q^2+2x-2 & q+x-1 & -q^2-x+1 & 0 & 0 & \dots & 0 & 0\\
    0 & 0 & 0 & q+x-1 & 2q+2x-2 & 1-q-x & 0 & 0 & \dots & 0 & 0\\
    0 & 0 & 0 & -q^2-x+1 & 1-q-x & q^2x+q^2+x-1 & -qx & 0 & \dots & 0 & 0\\
    \hline
     0 & 0 & 0 & 0 & 0 & -qx &  &  &  &  & \\
     0 & 0 & 0 & 0 & 0 & 0 &  &  &  &  & \\
     \vdots & \vdots & \vdots & \vdots & \vdots & \vdots & & & \Dq^{P_{n-6}}-xI & & \\
     0 & 0 & 0 & 0 & 0 & 0 &  &  &  &  & \\
     0 & -qx & 0 & 0 & 0 & 0 &  &  &  &  & \\
         \end{array}\right].}
\end{align*}

 We know that the determinant can be expressed in terms of permutations, so consider the possible terms of $\det(\Dq^G-xI)$. If the selection for the second column is not in the $[1,2]$ block, then the selection for the last column of the nonzero terms will be the second row. Similarly, if the selection for the sixth column is not in the $[3,4,5,6]$ block, then the selection for the seventh column of the nonzero terms will be the sixth row. 

 Therefore, we have four cases to consider.
 \begin{itemize}
     \item The selection for the second column is in the $[1,2]$ block and the selection for the sixth column is in the $[3,4,5,6]$ block.
     \item The selection for the second column is in the $[1,2]$ block and the selection for the sixth column is \textit{not} in the $[3,4,5,6]$ block. 
      \item The selection for the second column is \textit{not} in the $[1,2]$ block and the selection for the sixth column is in the $[3,4,5,6]$ block. 
       \item The selection for the second column is \textit{not} in the $[1,2]$ block and the selection for the sixth column is \textit{not} in the $[3,4,5,6]$ block.
 \end{itemize}
    
 The resulting term for each of the cases is (respectively)
 \begin{itemize}
     \item $(2q^2-q+x-1)(q+x-1)^3(4q^2x+2q^2-q+x-1)(2q^2x+2q^2-q+x-1)\mathcal{P}_{\Dq^{P_{n-6}}}(x)$
     \item $4(2q^2-q+x-1)(q+x-1)^3(2q^2x+2q^2-q+x-1)\mathcal{P}_{\Dq^{P_{n-7}}}(x)$
     \item $2(2q^2-q+x-1)(q+x-1)^3(4q^2x+2q^2-q+x-1)\mathcal{P}_{\Dq^{P_{n-7}}}(x)$
     \item $(2q^2-q+x-1)(q+x-1)^3(8q^4x^4)\mathcal{P}_{\Dq^{P_{n-8}}}(x)$
 \end{itemize}

\begin{figure}
    \centering
    \setlength{\tabcolsep}{24pt}
    \begin{tabular}{cc}
      \begin{tikzpicture}[scale=0.7]
      \node at (3,-0.5) {$P_{n-6}$};
        \node[vertex](0) at (0:0) {7};
        \node[vertex](1) at (0:9) {1};
        \node[vertex](2) at ({6+3*cos(60)},{3*sin(60)}) {2};
        \node[vertex](3) at (120:3) {3};
        \node[vertex](4) at (180:3) {4};
        \node[vertex](5) at (240:3) {5};
        \node[vertex](6) at (300:3) {6};
        \node[vertex](7) at (1.5,0) {};
        \node[vertex](8) at (4.5,0) {};
        \node[vertex](9) at (6,0) {$n$};
        
        \node at (3,0) {\dots};
        
        \draw[thick](9)--(1) (9) -- (2) (0) -- (3) (0) -- (4) (0) -- (5) (0) -- (6)(1) -- (2) (3)--(4) (5)--(6)  (0)--(7)--(2.5,0) (3.5,0)--(8)--(9);
    \end{tikzpicture}
      &\begin{tikzpicture}[scale=0.7]
        \node[vertex](0) at (0,0) {$4$};
        \node[vertex](1) at (2,2) {7};
        \node[vertex](2) at (0,2) {2};
        \node[vertex](3) at (0,-2) {6};
        \node[vertex](4) at (-2,-2) {5};
        \node[vertex](5) at (-2,0) {3};
        \node[vertex](6) at (-2,2) {1};
        \node[vertex](7) at (2,-2) {$n$};
        \node[vertex](8) at (2,1) {};
        \node[vertex](9) at (2,-1) {};
        
        \node at (2,0.2) {\vdots};
        \node at (3,0.2) {$P_{n-6}$};
        
        \draw[thick](0) -- (2) (0) -- (3) (0) -- (4) (0) -- (5) (0) -- (6)(4) -- (5)(3)--(5) (3)--(4) (5)--(6)(2)--(6)(2)--(5) (1)--(8)--(2,0.5) (2,-0.5)--(9)--(7);
    \end{tikzpicture}
     \\ 
      (a) & (b)   
    \end{tabular}
    \caption{(a) Graph $G$ with a path of length $n-6$ as an induced subgraph. (b) Graph $H$ with a path of length $n-6$ as a connected component. Only when $q=\frac12$, $G$ and $H$ are cospectral for the exponential distance matrix.}
    \label{fig:dq12-cospec-friendship}
\end{figure}
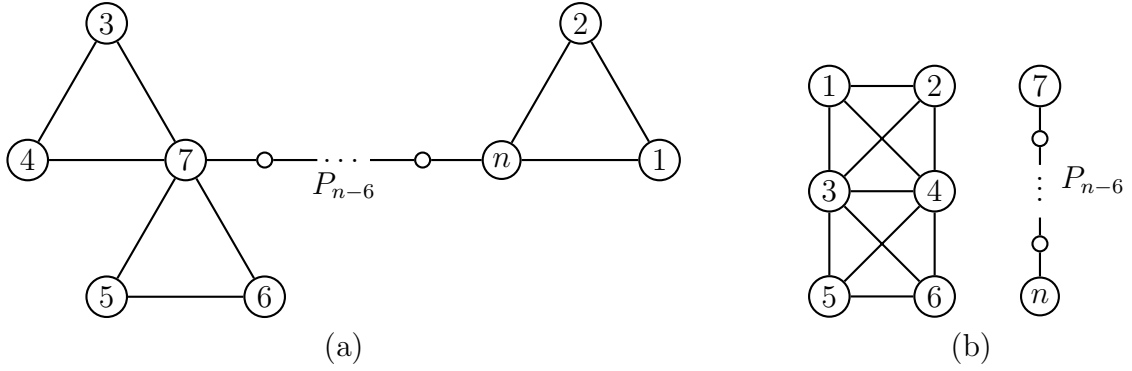

When we combine these equations, the signs of the second and third cases will be negative and the signs of the first and fourth cases will be positive. Therefore,
\small{\begin{align*}
\mathcal{P}_{\Dq^G}(x) &= (2q^2-q+x-1)(q+x-1)^3\Bigg[(4q^2x+2q^2-q+x-1)(2q^2x+2q^2-q+x-1) \cdot \mathcal{P}_{\Dq^{P_{n-6}}}(x) 
\\
&- 8q^2x^2\left(\left(2q^2x+\frac{3}{2}q^2-\frac{3}{4}q+\frac{3}{4}x-\frac{3}{4}\right) \cdot \mathcal{P}_{\Dq^{P_{n-7}}}(x) - x^2q^2 \cdot \mathcal{P}_{\Dq^{P_{n-8}}}(x)\right)\Bigg].
\end{align*}
}

When $q = \frac{1}{2}$, we can apply Proposition \ref{prop:charpoly-recursive} so that
\begin{align*}
   \left(2q^2x+\frac{3}{2}q^2-\frac{3}{4}q+\frac{3}{4}x-\frac{3}{4}\right) \cdot \mathcal{P}_{\Dq^{P_{n-7}}}(x) - q^2x^2 \cdot \mathcal{P}_{\Dq^{P_{n-8}}}(x) \\
   = 
    \big((q^2+1)x-1+q^2\big) \cdot \mathcal{P}_{\Dq^{P_{n-7}}}(x) - q^2x^2 \cdot \mathcal{P}_{\Dq^{P_{n-8}}}(x).
\end{align*}

Thus, we can only factor out the characteristic polynomial of the path when $q = \frac12$. So the characteristic polynomial of $G$ is
\begin{align*}
\mathcal{P}_{\Dq^G}(x) =& (2q^2-q+x-1)(q+x-1)^3 \\ 
&\cdot \Big[(4q^2x+2q^2-q+x-1)(2q^2x+2q^2-q+x-1) - 8q^2x^2\Big] \cdot \mathcal{P}_{\Dq^{P_{n-6}}}(x).
\end{align*} 

When $q = \frac12$, we verify that $\mathcal{P}_{\Dq^G}(x) = \mathcal{P}_{\Dq^H}(x)$.

\end{proof}

\begin{figure}
    \centering
    \setlength{\tabcolsep}{24pt}
    \begin{tabular}{cc}
        \begin{tikzpicture}
            \node at (3,-0.5) {$P_{n-6}$};
        
            \node[vertex](0) at (6,1) {1};
            \node[vertex](1) at (72:1) {3};
            \node[vertex](2) at (144:1) {4};
            \node[vertex](3) at (216:1) {5};
            \node[vertex](4) at (282:1) {6};
            \node[vertex](5) at (1,0) {7};
            \node[vertex](6) at (2,0) {};
            \node[vertex](7) at (4,0) {};
            \node[vertex](8) at (5,0) {$n$};
            \node[vertex](9) at (6,-1) {2};
        
            \node at (3,0) {\dots};
        
            \draw[thick](1)--(2) (1)--(3) (1)--(4) (1)--(5) (2)--(3) (2)--(4) (2)--(5) (3)--(4) (3)--(5) (4)--(5) (5)--(6) (6)--(2.5,0) (3.5, 0)--(7) (7)--(8) (8)--(9) (8)--(0);

            \node at (0,-1.45) {};
        \end{tikzpicture}
        & \begin{tikzpicture}
            \node[vertex](0) at (0:1) {5};
            \node[vertex](1) at (60:1) {6};
            \node[vertex](2) at (120:1) {1};
            \node[vertex](3) at (180:1) {2};
            \node[vertex](4) at (240:1) {3};
            \node[vertex](5) at (300:1) {4};

            \node[vertex](6) at (2,1.5) {7};
            \node[vertex](7) at (2,0.75) {};
            \node[vertex](8) at (2,-0.75) {};
            \node[vertex](9) at (2,-1.5) {$n$};
        
            \node at (2,0.1) {\vdots};
            \node at (2.5, 0) {$P_{n-6}$};
        
            \draw[thick](0)--(1) (0)--(2) (0)--(4) (0)--(5) (1)--(2) (1)--(3) (1)--(4) (1)--(5) (2)--(3) (2)--(4) (2)--(5) (3)--(4) (3)--(5) (4)--(5) (6)--(7)--(2, 0.3) (2, -0.3)--(8)--(9);
        \end{tikzpicture} \\
        (a) & (b)   
    \end{tabular}
    \caption{(a) Graph $G$ with a path of length $n-6$ as an induced subgraph. (b) Graph $H$ with a path of length $n-6$ as a connected component. Only when $q=\frac12$, $G$ and $H$ are cospectral for the exponential distance matrix.}
    \label{fig:dq12-cospec-complete-star}
\end{figure}
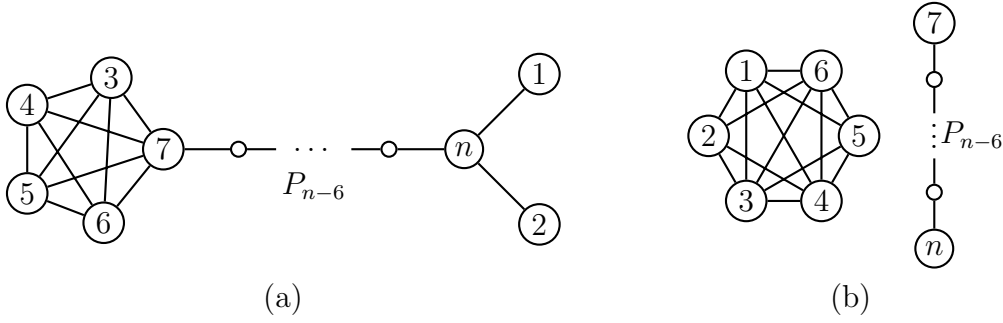

After proving cospectrality of the family in Theorem \ref{box-windmill cospec}, we found another family that is uniquely cospectral with a path extension when $q = \frac12$.

\begin{thm}
    The graphs $G$ and $H$ in Figure \ref{fig:dq12-cospec-complete-star} are cospectral on the exponential distance matrix for $q = \frac{1}{2}$.
\end{thm}

\begin{proof}
    The proof of cospectrality of this family mirrors the proof of Theorem \ref{box-windmill cospec}. The notable difference is within the characteristic polynomials. 
    The characteristic polynomial of $H$ is 
\begin{align*}
\mathcal{P}_{\Dq^H}(x) = (q^2+x-1)(q+x-1)^3 (3q^3 - q^2x - 7q^2 - 3qx + x^2 + 3q - 2x + 1) \mathcal{P}_{\Dq^{P_{n-6}}}(x).
\end{align*}

We can follow the same partition, row operations, and permutation selection procedure as described in the proof of Theorem \ref{box-windmill cospec}. Our resulting characteristic polynomial will be
\begin{align*}
\mathcal{P}_{\Dq^G}(x) = (q^2+x-1)(q+x-1)^3\Bigg[(4q^2x+4q^2-3q+x-1)(2q^2x+q^2+x-1) \cdot \mathcal{P}_{\Dq^{P_{n-6}}}(x) 
\\
- 8q^2x^2\left(\left(2q^2x+\frac{3}{2}q^2-\frac{3}{4}q+\frac{3}{4}x-\frac{3}{4}\right) \cdot \mathcal{P}_{\Dq^{P_{n-7}}}(x) - x^2q^2 \cdot \mathcal{P}_{\Dq^{P_{n-8}}}(x)\right)\Bigg].
\end{align*}

As in Theorem \ref{box-windmill cospec}, we can use the recursive formula of a path from \cite{expdist} to combine the characteristic polynomials $\mathcal{P}_{\Dq(P_{n-7})}(x)$ and $\mathcal{P}_{\Dq(P_{n-8})}(x)$. As in the previous proof, this combination only occurs when $q = \frac12$.
So, we write
\begin{align*}
   \left(2q^2x+\frac{3}{2}q^2-\frac{3}{4}q+\frac{3}{4}x-\frac{3}{4}\right) \cdot \mathcal{P}_{\Dq^{P_{n-7}}}(x) - q^2x^2 \cdot \mathcal{P}_{\Dq^{P_{n-8}}}(x)
   \\
   =  \big((q^2+1)x-1+q^2\big) \cdot \mathcal{P}_{\Dq^{P_{n-7}}}(x) - q^2x^2 \cdot \mathcal{P}_{\Dq^{P_{n-8}}}(x)
    = \mathcal{P}_{\Dq^{P_{n-6}}}(x)
\end{align*}

 Therefore, assuming $q = \frac12$, we can rewrite the characteristic polynomial of graph $G$ as 
 
\small{\begin{align*}
\mathcal{P}_{\Dq^G}(x) = (q^2+x-1)(q+x-1)^3 \Big[(4q^2x+4q^2-3q+x-1)(2q^2x+q^2+x-1) - 8q^2x^2\Big] \mathcal{P}_{\Dq^{P_{n-6}}}(x).
\end{align*}
}

When $q=\frac12$, $\mathcal{P}_{\Dq^G}(x) = \mathcal{P}_{\Dq^H}(x)$, so these two families are cospectral when $q=\frac12$.    
\end{proof}

\section{Future Work}

A question that still remains to be answered that we proposed in Section \ref{sec:cospecforallq} is: ``When does cospectrality hold for all values of $q$?'' For graphs on a small number ($\leq 9$) vertices, the answer appears to be independent of the graph structure. 
\begin{conj}
    If graphs are cospectral for two different $q$ values (not $0$ or $1$) using the same similarity matrix, then they are cospectral for all values of $q$.
\end{conj}
 This conjecture holds for diameter $2$ graphs with our results. 

 If a pair is only cospectral for exactly one (non-trivial) value of $ 0 \leq q \leq 1$, on a small number of vertices, then $q=\frac12$. Expanding the possible values of $q$, this behavior also happens when $q=-2,3,\pm 2i$. Further work and understanding into these numbers and their relation to special cospectral pairs is needed.

 Finally, our switching cospectral construction dramatically increases the number of matrices Godsil-McKay Switching applies to. A remaining class of matrices is the distance Laplacian, distance signless Laplacian, and distance normalized Laplacian. All of the matrices can be written in the form $M=\alpha \operatorname{diag}(\sum_{j=1}^n\dist(v_i,v_j))+\beta \D$. Currently, our construction does not apply to this family. Expanding our construction or Godsil-McKay switching to these Laplacian variations would help us understand patterns of cospectrality for more matrices. 

\section*{Acknowledgment}
This research was done primarily at the 2024 Iowa State University Math REU, which was supported through NSF Grant DMS-1950583.

\bibliographystyle{plain} 

\end{document}